\documentclass{article}

\usepackage{PRIMEarxiv}
\usepackage[utf8]{inputenc} 
\usepackage[T1]{fontenc}    
\usepackage{hyperref}       
\usepackage{url}            
\usepackage{booktabs}       
\usepackage{amsfonts}       
\usepackage{nicefrac}       
\usepackage{microtype}      
\usepackage{fancyhdr}       
\usepackage{graphicx}       
\graphicspath{{media/}}     
\usepackage{math}
\usepackage[toc,page]{appendix}
\usepackage{math}
\usepackage{enumitem}
\usepackage{algorithm}
\usepackage{algpseudocode}
\usepackage{comment}

\newtheorem{example}[theorem]{Example}

\title{\LARGE \bf
Characterizing Controllability and Observability for Systems with Locality, Communication, and Actuation Constraints
}

\author{\parbox{6. in}{\centering Lauren Conger\qquad  Yiheng Lin\qquad  Adam Wierman\qquad  Eric Mazumdar
         \thanks{LC is supported by a NDSEG fellowship from the Air Force Office of Research and a PIMCO fellowship. YL is supported by Amazon AI4Science Fellowship and PIMCO Graduate Fellowship in Data Science. This work is supported by NSF Grants CNS-2146814, CPS-2136197, CNS-2106403, NGSDI-2105648, CNS-2240110.}\\
         Caltech\\
         {\tt\small \{lconger,yihengl,adamw,mazumdar\}@caltech.edu}}
         }
\begin{document}
\maketitle
\begin{abstract}
This paper presents a closed-form notion of controllability and observability for systems with communication delays, actuation delays, and locality constraints. The formulation reduces to classical notions of controllability and observability in the unconstrained setting. As a consequence of our formulation, we show that the addition of locality and communication constraints may not affect the controllability and observability of the system, and we provide an efficient sufficient condition under which this phenomenon occurs. This contrasts with actuation and sensing delays, which cause a gradual loss of controllability and observability as the delays increase. We illustrate our results using linearized swing equations for the power grid, showing how actuation delay and locality constraints affect controllability.
\end{abstract}

\section{Introduction}
Sensor and actuator placement is a crucial problem in the design of real world systems in application areas ranging from power systems and traffic networks to environmental monitoring. Taking power systems as an example, devices such as generators and transformers can provide additional inputs to actuate the system, while other devices like meters generate signals about the system's state. A system designer may face the challenge of deploying a limited number of such devices to different locations in a power system with the goal of keeping the system stable or performing well in other control tasks. While sensor and actuator placement is an old problem that has been well studied \cite{kalman_controllability_1961,gilbert_controllability_1963}, the proliferation of low-cost sensors and the decentralization of control in large systems like power systems has led to new challenges.
Many real-world systems are subject to delayed and information-constrained sensors that can significantly impact the quality of the control. However, we currently lack tools to rigorously characterize the performance loss under various sensor configurations. This leads us to study the problem of quantifying the quality of a set of controllers and sensors in systems subject to actuation and sensing delays as well as communication constraints.

While task-specific metrics like cost are commonly used to measure the quality of a set of sensors and actuators, there are many settings where the cost is unknown or time varying; here task-agnostic notions of controllability and observability of dynamical systems~\cite{gilbert_controllability_1963,georges_use_1995} offer a suitable proxy. Controllability, defined as the determinant of the controllability Gramian \cite{kailath_linear_1980}, is the volume over which a control input with magnitude one can drive the state to the zero state. Likewise, observability, defined as the determinant of the observability Gramian, is the volume of the signal-to-noise ratio along any direction in the state space. %
These metrics are used to compare different sets of actuators and sensors for dynamical systems; a system with a larger controllability, given fixed dynamics, is preferred to one with smaller controllability. %

Although the notions of controllability and observability are well-defined for classic control problems, new challenges arise from the decentralized control of modern large-scale systems \cite{shin2023near,lin2022decentralized, li2021distributed,yu_online_2023}, where the controllers may face different constraints on communication speed, actuation delays, and/or locality. Such constraints, which we refer to as \textit{system level constraints}, cannot easily be captured in the classical formulations of controllability/observability based on the Gramian or control inputs \cite{kailath_linear_1980}. Fortunately, a recently proposed framework called \textit{system level synthesis (SLS)} \cite{anderson_system_2019} handles system level constraints efficiently, where the key technique is to formulate the original control problem in terms of closed-loop maps. Some recent works address controllability and observability in the specific setting of systems with delays \cite{wang_method_2020,zhou_data-driven_2018,jinna_controllability_2009}; SLS allows for a much broader class of constraints, including delays. Using SLS, a recent work \cite{li_global_2023} shows sufficient conditions under which locality constraints on model predictive control (MPC) leads to \textit{no loss in optimality} for any convex global cost function. However, even with SLS, it is unclear whether one can formally extend the notion of controllability and observability to include system level constraints faced by large-scale networked systems, and whether such measures can be computed efficiently to quantify the loss of performance under different combinations of system level constraints. %

\textbf{Contributions.} In this work, we propose a formal definition of controllability and observability subject to system level constraints (Section \ref{sec:constrained_systems}). Our definitions are based on the SLS framework and we draw connections to classical distributed control, differential privacy, and speed-accuracy tradeoffs in neurons. We provide a provable closed-form solution for our SLS-based definitions (\Cref{alg:ellipsoid-volume}, \Cref{prop:ellipsoid-algorithm}); the definitions strictly generalize the classic notions of controllability and observability (\Cref{thm:reduction}). 

Importantly, we establish rank conditions for system level constraints under which there is no performance loss compared with the unconstrained system (\Cref{thm:optimality_constrained_systems}). This rigorously characterizes the empirical observation that locality and communication constraints do not necessarily decrease controllability and observability. In contrast, actuation delay does not satisfy this condition, consistent with numerical results showing a gradual decrease in controllability as actuation delay increases. We perform numerical experiments on a synthetic example and a real-world model of linearized and discretized swing equations for the power grid (\Cref{sec:numerics}).

\textbf{Notation.}
Let $N_x\in\mathbb{Z}_+$ and $N_u\in\mathbb{Z}_+$ be the state and input dimensions respectively. Bold letters such as $\mathbf{x}$ indicate signals. Let $x_0\coloneqq x(0)$ and $x_T\coloneqq x(T)$. A single index such as $Z[i]$ indicates the $i^{th}$ row or block of rows of a matrix, and a double index such as $Z[i,j]$ indicates the $i^{th}$ row and $j^{th}$ column of a matrix. Brackets around a scalar indicates the set of indices $[a]\coloneqq [1,\dots,a]$. A vector over a variable, such as $\vec\phi_x$, is the vectorization of matrix, where the columns are stacked into a vector. A hat symbol over a matrix indicates the vectorization of matrix multiplication, that is, $\hat A$ is defined such that if $Y=A X$ for two matrices $X$ and $Y$, then $\vec y = \hat A \vec x$. We define $\supp A$ to be an indicator matrix such that $(\supp A)[i, j] = 0$ if $A[i, j]$ is constrained to be zero.

\section{Motivation \& Background}
We begin by introducing the problem of computing the controllability and observability of a dynamical system and illustrating the challenges that come from considering constraints in the classical formulation of the problem. We then show how system level synthesis (SLS) offers a framework for computing the controllability and observability of a system while incorporating system-level constraints.

\subsection{Classical Controllability and Observability}
We consider linear systems with dynamics given by
\begin{align}\label{eq:dynamics}
    x(t+1) = A x(t) + B u(t)\,, \quad
    y(t) = C x(t)\,,
\end{align}
where $x(t) \in \mathbb{R}^{N_x}$, $u(t) \in \mathbb{R}^{N_u}$, and $y(t) \in \mathbb{R}^{N_y}$. A system is said to be \textit{controllable} in time $T$ if it can be driven from any state $x_0$ to any terminal state $x_T$. For an unconstrained system, the controllability Gramian and controllability are given, respectively, by
\begin{align}\label{eq:controllability_classical}
   W_c = \sum_{i=0}^{T-1} A^i B B^\top (A^\top)^i\,, \quad V_c = \det W_c\,.
\end{align}
The magnitude of each eigenvalue of $W_c$ is the distance in the direction of the eigenvector that one unit of control provides (after accounting for the trajectory resulting from the initial condition alone). The determinant of $W_c$ is then the state space volume that is reachable with one unit of control input. However, this computational method assumes that any sequence of control input can be selected; this assumption is not true when information among states is restricted, actuation is delayed, or controllers are dependent on only local information.

Similarly, for an unconstrained system, one can define the observability Gramian and observability respectively, as 
\begin{align}\label{eq:observability_classical}
    W_o = \sum_{i=0}^{T-1} A^{\top^i} C^\top C A^i\,, \quad V_o = \det W_o\,.
\end{align}
In this case, the magnitude of each eigenvalue of the Gramian $W_o$ is the signal-to-noise ratio for a state in the direction of the corresponding eigenvector. The determinant of $W_o$ corresponds to the `observability volume' of the system. Similarly to controllability, this method assumes that none of the observations are delayed and that state estimation uses all of the observation estimation. This does not apply to large-scale systems where state estimation depends only on local observation, and the observations are not immediate. 

This paper proposes a new method for computing controllability and observability that accounts for delayed actuation and observation, as well as local control and information sharing. Critical to the proposed approach is the system level synthesis (SLS) parameterization of the linear system, which we outline in the following section. In what follows, we consider the state estimation problem as the dual of the control problem; in what follows, one can substitute 
\begin{align*}
    A \leftrightarrow A^\top\,, \quad B \leftrightarrow C^\top \,.
\end{align*}
to obtain the corresponding observability framework.

\subsection{System Level Synthesis (SLS)}
The SLS parametrization of the state dynamics in \eqref{eq:dynamics} can be written as 
\begin{align}\label{eq:SLS_dynamics}
    \phi_x(t+1) = A \phi_x(t) + B \phi_u(t)\,,
\end{align}
where $x(t)=\phi_x(t) x_0$ and $u(t)=\phi_u(t) x_0$, $\phi_x(t)\in\R^{N_x \times N_x}$, $\phi_u(t)\in\R^{N_u \times N_x}$, and $\phi_x(0)=I$. Let the closed-loop maps $\phi_x$, $\phi_u$, and $\phi$ be defined as 
\begin{align*}
    \phi_x =\begin{bmatrix}
        \phi_x(0) \\ \vdots \\  \phi_x(T\tminus1)
    \end{bmatrix}\,, \ \ 
    \phi_u = \begin{bmatrix}
        \phi_u(0) \\ \vdots \\ \phi_u(T\tminus1)
    \end{bmatrix}\,, \ \
        \phi = \begin{bmatrix}
            \phi_x \\ \phi_u
        \end{bmatrix}\,.
\end{align*}

A prior result from~\cite[Theorem 2]{anderson_system_2019} guarantees that the subspace defined by $\phi$ parameterizes all possible system responses and that the state feedback controller is $u(t) = K(t) x(t)$ with $K(t)=\phi_u(t) \phi_x(t)^{-1}$;
therefore optimizing over $\phi$ is equivalent to optimizing over $u$.

This parameterization allows us to explicitly write constraints on $\phi_x$ and $\phi_u$ which encode the communication, actuation, and sensing constraints. 
To enforce that the $i^{th}$ entry of $x_0$ does not affect the $j^{th}$ entry of the state or input at time $t$, we set $\phi_x(t)[j,i]=0$ or $\phi_u(t)[j,i]=0$ respectively. We write the system level constraints as
\begin{align}\label{eq:constraints}
    S_x \vec \phi_x = 0\,, \quad S_u \vec\phi_u = 0\,,
\end{align}
where $S_x$ ($S_u$ resp.) has a single nonzero entry in each row, with the column corresponding to the entry of $\phi_x$ ($\phi_u$ resp.) that is constrained to be zero. We define the set of closed-loop maps that satisfy \eqref{eq:constraints} as $\mathcal{S}$, such that $\phi\in \mathcal{S}$. Equipped with the ability to specify system level constraints, in the next section we present the definitions of controllability under these constraints.

\section{Constrained Systems}\label{sec:constrained_systems}
Our focus in this paper is systems that are subject to actuation delay, communication delay and restrictions, and locality constraints. These constraints can be expressed as support constraints on the closed-loop maps $\phi_x$ and $\phi_u$ introduced in the previous section. 

We formally define the different types of constraints considered in this work as follows. We define a directed graph whose adjacency matrix $\G[i,j]$ is defined such that $\G[i,j]=1$ if $A[i,j]\ne 0$; $\G[i, j] = 0$ otherwise. 
\begin{itemize}[nolistsep, leftmargin=*]
\item \textbf{Actuation delay}: The system is subject to a $\tau$-step actuation delay if the actuator can react to state information at least $\tau$ steps after, i.e., $\phi_u(t) = 0,$ for $t < \tau$. 
\item \textbf{Communication speed}: The communication speed is $v$ if the propagation of communication signals is $v$ times the dynamics, i.e., $\supp{\phi_x(t)} = \supp{\G^{\lfloor vt \rfloor}}$ and $\supp \phi_u(t) = \supp\left( \abs{B^\top} \supp \phi_x(t) \right)$, so that $x(t)[j]$ cannot depend on $x_0[i]$ if $\text{dist}_{\mathcal{G}}(i, j) > vt$. Likewise, $u(t)[i]$ cannot depend on $x_0[j]$ if $B[k,i]=0$ for all $k$ such that $\text{dist}_\G(j,k) \le vt$.
\item \textbf{Locality}: The system is subject to a locality constraint with radius $r$ if the state of a node can only be affected by the states of its $r$-hop neighbors, i.e., $\phi_x(t)[j, i] = 0$ if $\text{dist}_\G(i, j) > r$, and a controller uses information from the nearest $r$-hops away from the states on which it acts, specified as $\supp{\phi_u} = \supp \left(\abs{B^\top} \supp{\phi_x} \right)$.
\end{itemize}

We formulate the definition of controllability and observability for constrained systems using SLS operators. We consider the magnitude of any $x_0$ such that one unit of control can drive the state to the origin, again scaled by the effect of the dynamics without any control.
Using this notion of controllability, we define the controllability of \eqref{eq:dynamics} with constraints \eqref{eq:constraints} as the following:
\begin{equation}\label{eq:controllability_SLS}
    \begin{split}
            Y_c ={}& \text{volume}\left\{ A^T x_0\ :\ h_c(x_0)\le 1\right\}, \text{ where } \\
    h_c(x_0) \coloneqq{}& \min_{\phi_u} x_0^\top \left(\sum_{i=0}^{T-1}\phi_u(i)^\top \phi_u(i) \right)x_0  \\
    &\text{s.t.}\ S_c \vec \phi_u + C_c = 0\,.
    \end{split}
\end{equation}
The constraint for the terminal state condition, $x_T=0$, as well as support constraints for $\phi_x$ and $\phi_u$, are specified by the form of $S_c$ and $C_c$. Recall that we use the hat symbol over a matrix to indicate the vectorization of matrix multiplication. The constraint matrices $S_c$ and $C_c$ are given by %
\begin{align*}
    S_c &=\begin{bmatrix}
        S_u \\ D_1 \\ D_2
    \end{bmatrix},\
    C_c = \begin{bmatrix}
        0 \\ v_1 \\ \hat A ^T \vec I
    \end{bmatrix},\ v_1 =  \begin{bmatrix}
       S_x[1]  \vec I \\
       \vdots \\
        S_x[T\tminus1] \hat A^{T\tminus1} \vec I
    \end{bmatrix}\,, \ 
    D_1 = \begin{bmatrix}
        S_x[1] \hat B & & &\\
        S_x[2] \hat A \hat B & S_x[2] \hat B & & \\
        \vdots & & \ddots & \\
        S_x[T\tminus1] \hat A ^{T\tminus2} \hat B & \dots & & \hat B 
    \end{bmatrix}\,,\\
    D_2 &= \begin{bmatrix}
        \hat A ^{T\tminus1}\hat B & \dots & \hat B
    \end{bmatrix}\,.
\end{align*}
The observability of \eqref{eq:dynamics} is likewise given by
\begin{equation}\label{eq:observability_SLS}
    \begin{split}
        Y_o ={}& \text{volume}\left\{ A^{T^\top} x_0\ :\ h_o(x_0)\le 1\right\}, \text{ where }\\
    h_o(x_0) \coloneqq{}& \min_{\phi_u} x_0^\top \left(\sum_{i=0}^{T-1}\phi_u(i)^\top \phi_u(i) \right)x_0  \\
    &\text{s.t.}\ S_o \vec \phi_u + C_o = 0\,,
    \end{split}
\end{equation}    
with $S_o$ and $C_o$ again utilizing the terminal constraint $x_T=0$, with $S_o$ and $C_o$ defined analogously to $S_c$ and $C_c$.

\begin{example}[Distributed Control]\label{example:distributed} 
    In large-scale systems with fast propagation of the dynamics, such as the power grid, the speed at which sensors send information can be on the same order of magnitude as the dynamics \cite{taft_sensing_2016} and understanding controllability and observability with delays is an important problem \cite{nwankpa_observability_2015}. This means that information about a distant state may not be available for a number of time steps, and controllers must be designed accounting for this lack of information. Additionally, if a large disturbance hits a system, ensuring that the disturbance will not propagate further than $r$ neighbors ensures that distant states will not be affected by a disturbance about which they only have delayed information \cite{anderson_system_2019}. The sensor and actuator co-design problem with constraints \cite{matni_regularization_2014} is important for addressing these new challenges; comparing the controllability and observability of sensors and actuators for large-scale systems with these delays and constraints helps designers understand how much faster or more densely-placed hardware improves performance.
\end{example}
\begin{example}[Sensorimotor control]\label{example:sensorimotor} In neurophysiology, a nerve is made up of several heterogeneous communication channels formed by different types of axons. To achieve the control objective, the design of the communication channels faces a strict speed-accuracy trade-off \cite{nakahira2019theoretical} over channels with lower delay and less information, or a channel with longer delay and more information. Our constrained observability problem \eqref{eq:observability_SLS} provides a method to compare the observability of different sets of axons that satisfy a fixed metabolic cost.
\end{example}

\begin{example}[Privacy as Observability]\label{example:privacy}
The problem of preserving privacy in decentralized control of large-scale networked systems has received considerable attention recently, including privacy through the perspective of observability \cite{zhang_observability_2023,kawano_design_2020,kawano_revisit_2018,degue_differentially_2020,le_ny_privacy-preserving_2015}. Using the interpretation of observability as a signal-to-noise ratio of an observer, less observability potentially corresponds to less noise needed for obtaining some fixed level of privacy. In particular, our method provides tools for analyzing this signal-to-noise ratio with sensing delay and locality of sensors.
 
\end{example}

\section{Results}

This section presents our main three technical results. Theorem~\ref{thm:reduction} states that the controllability volume given by \eqref{eq:controllability_SLS} is equivalent to classical controllability \eqref{eq:controllability_classical} when the constraint set is empty. Algorithm~\ref{alg:ellipsoid-volume} offers a provably correct closed-form solution for the controllability, and Theorem~\ref{thm:optimality_constrained_systems} provides conditions under which adding constraints to the closed-loop map $\phi$ does not change the set of feasible states and inputs. As a consequence, the controllability does not change with the addition of the constraints (Corollary \ref{cor:rank_condition}).

\subsection{Connection to classical methods}

Our first result shows that, in the setting where there are no support constraints, the generalized definition of controllability \eqref{eq:controllability_SLS} reduces to the classical definition \eqref{eq:controllability_classical}. This shows that the formulation is consistent with classical  controllability.

\begin{theorem}\label{thm:reduction}
    When $S_u=0$ and $S_x=0$, $Y_c^2=V_c$ for all $A\in\R^{N_x \times N_x}$, $B \in \R^{N_x \times N_u}$.
\end{theorem}

\begin{proof}
    When the constraint set is empty, that is, when $S_x$ and $S_u$ are zero, the remaining constraint in $S_c$ and $C_c$ is the terminal constraint $x_T=0$.
    Using $u=\phi_u x_0$, we can rewrite the objective of \eqref{eq:controllability_SLS} as $\norm{u}^2$. The constraint $S_c \vec{\phi}_u + C_c = 0$ can be written as $A^T x_0 + D u = 0$. Here, $D$ takes the form $D=[A^{T\tminus1}B\ \dots B]$ and $u=[u(0),\dots,u(T\tminus1)]$. The optimal $u$ is given by
    \begin{align*}
        u = -D^\dagger A^T x_0\,.
    \end{align*}
    Plugging this into $h_c(x_0)$ results in 
    \begin{align*}
        h_c(x_0) = x_0^\top \left(A^{T^\top} D^{\dagger^\top}  D^\dagger A^T  \right) x_0 =  \norm{D^\dagger A^T x_0}_2^2\,.
    \end{align*}
    To compute the volume of all $x_0$ such that $h_c(x_0) \le 1$, note that the volume in terms of a new variable $y=A^T x_0$ is the volume of an ellipse defined by $D^\dagger$:
    \begin{align*}
        \text{volume}\bigg\{A^T x_0\ \bigg|\ h_c(x_0)\le 1 \bigg\} = \text{volume}\bigg\{y \bigg| \norm{D^\dagger y}^2 \le 1 \bigg\}  =\sqrt{\det D D^\top}\,,
    \end{align*}
    and using the definition of $D$ results in
    $Y_c^2 = \sum_{i=0}^{T-1} A^i B B^\top (A^\top)^i = V_c\,.$
\end{proof}

\subsection{Computing Constrained Controllability}

Our second contribution is a novel algorithm (Algorithm \ref{alg:ellipsoid-volume}) to calculate controllability with constraints in closed form.

\begin{algorithm}
\caption{Constrained Controllability}\label{alg:ellipsoid-volume}
\begin{algorithmic}[1]
\Require Horizon $T$; Constraint matrices $S_c$ and $C_c$.
\State Decompose $S_c \vec{\phi}_u + C_c = 0$ as
$S_i \Phi_i + C_i = 0,\ i \in [N_x].$\label{alg:ellipsoid-volume:s1}
\Comment{$\Phi_i$ contains the entries in $\vec{\phi}_u$ that correspond to $x_0[i]$.}
\State Compute the general solution $\Phi_i = \omega_i + V_i \cdot \lambda_i, \ i \in [N_x]$.\label{alg:ellipsoid-volume:s2}
\Comment{$\lambda_i \in \mathbb{R}^{r_i}$ and $V_i \in \mathbb{R}^{(TN_u) \times r_i}$ has full column rank.}
\State Define $M = [V_1, V_2, \cdots, V_n] \in \mathbb{R}^{(TN_u) \times (\sum_{i = 1}^{N_x}r_i)}$. \label{alg:ellipsoid-volume:s3}
\State Define $N = [\omega_1, \omega_2, \cdots, \omega_n] \in \mathbb{R}^{(TN_u) \times N_x}$. \label{alg:ellipsoid-volume:s4}
\State Return the volume \label{alg:ellipsoid-volume:s5}
\[\det(A^T) / \sqrt{\det((- M M^\dagger N + N)^\top (- M M^\dagger N + N))}.\]
\end{algorithmic}
\end{algorithm}

Algorithm \ref{alg:ellipsoid-volume} starts by eliminating the linear constraints in \eqref{eq:controllability_SLS} to reformulate it as an unconstrained optimization problem. Specifically, we first decompose $S_c \vec{\phi}_u + C_c = 0$ as $n$ independent sets of linear constraints $S_i \Phi_i + C_i = 0, \text{ for } i \in [N_x]$ in Step \ref{alg:ellipsoid-volume:s1}. 
Here, $\Phi_i \in \mathbb{R}^{T N_u}$ contains the entries in $\vec{\phi}_u$ that correspond to $\phi_u(t)[i, :]$ for $t < T$. Such decomposition can always be implemented due to the special structure of the support constraints induced by actuation delay, communication speed, and locality constraints. We state this result formally in Lemma \ref{lemma:independent-constraints}.

\begin{lemma}\label{lemma:independent-constraints}
For any $k$, the constraint $(S_c \vec{\phi}_u + C_c)[k] = 0$ cannot depend on entries from both $\Phi_i$ and $\Phi_j$ ($i \not = j$).
\end{lemma}

\begin{proof}
Under the definitions of constraints on actuation delay, communication speed, and locality in \Cref{sec:constrained_systems}, the constraint $(S_c \vec{\phi}_u + C_c)[k] = 0$ corresponds to either $\vec{\phi}_u[\tau] = 0$ or $\vec{\phi}_x[\tau] = 0$ for some index $\tau \in [TN_u N_x]$. In the first case, it satisfies the statement of \Cref{lemma:independent-constraints}.

Otherwise, $(S_c \vec{\phi}_u + C_c)[k] = 0$ must correspond to $\vec{\phi}_x[\tau] = 0$ for some index $\tau \in [TN_u N_x]$. Suppose it is converted from $\phi_x(t)[j, i] = c$, where $c \in \{0, 1\}$. By \eqref{eq:SLS_dynamics}, we see that $\phi_x(\tau + 1)[:, i] = A \phi_x(\tau)[:, i] + B \phi_u(\tau)[:, i]$ for all $\tau < t$. Therefore, $\phi_x(t)[j, i]$ can be expressed as a linear combination of $\{\phi_u(\tau)[:, i]\mid 0 \leq \tau < t\}$. Since all entries in $\phi_u(\tau)[:, i]$ are in $\Phi_i$, $(S_c \vec{\phi}_u + C_c)[k] = 0$ only involves entries from $\Phi_i$.
\end{proof}

After the decomposition, Step \ref{alg:ellipsoid-volume:s2} solves the general solution $\Phi_i = \omega_i + V_i \cdot \lambda_i$ of each constraint set $S_i \Phi_i + C_i = 0$, so we can change the variables from $\vec{\phi}_u$ to $\{\lambda_i\}_{i \in [N_x]}$ in the optimization problem \eqref{eq:controllability_SLS}. With the matrices $M$ and $N$ defined in Steps \ref{alg:ellipsoid-volume:s3}-\ref{alg:ellipsoid-volume:s4}, we see $h_c(x_0)$ can be expressed as
\begin{align}\label{equ:unconstrained-volume}
    \min_{\lambda_1, \ldots, \lambda_{N_x}} \norm{N x_0 + M (x_0[1] \lambda_1, \ldots, x_0[N_x] \lambda_{N_x})^\top}_2^2.
\end{align}

Finally, we conclude the discussion with a proposition that shows the output of Algorithm \ref{alg:ellipsoid-volume} is equal to the constrained controllability $Y_c$ defined in \eqref{eq:controllability_SLS}.

\begin{proposition}\label{prop:ellipsoid-algorithm}
The constrained controllability $Y_C$ satisfies $Y_c = \text{volume}\left\{A^T x_0 : \norm{(- M M^\dagger N + N)x_0}_2^2 \leq 1\right\}$, which is the output of Algorithm \ref{alg:ellipsoid-volume}.
\end{proposition}
\begin{proof}
By \eqref{equ:unconstrained-volume}, we see $h_c(x_0) = \min_{y} \norm{N x_0 + M y}_2^2$ holds for all $x_0$ that satisfies $x_0[i] \not = 0$ for all $i \in [N_x]$, because we can do the change of variables
\[y \coloneqq (x_0[1] \lambda_1, \ldots, x_0[N_x] \lambda_{N_x})^\top \in \mathbb{R}^{\sum_{i=1}^{N_x} r_i}.\]
Since this equation holds almost everywhere, we obtain that
\begin{align*}
    Y_c ={}& \text{volume}\left\{A^T x_0 : \min_{y} \norm{N x_0 + M y}_2^2 \leq 1\right\}\\
    ={}& \text{volume}\left\{A^T x_0 : \norm{(- M M^\dagger N + N)x_0}_2^2 \leq 1\right\},
\end{align*}
which is equal to the return value of Algorithm \ref{alg:ellipsoid-volume}.
\end{proof}

It is worth noting that the time complexity of \Cref{alg:ellipsoid-volume} can grow polynomially in time horizon $T$ and system dimensions $(N_x, N_u)$. In Section~\ref{subsec:computation_speed}, we show that the rank condition is faster than the volume computation, with the gap increasing as the state size increases.

\subsection{Maintaining Performance Under Constraints}
An important question for system design is whether the performance of the unconstrained system is preserved after imposing additional constraints on actuation delay, communication speed, and locality. While a straightforward answer is to compute the constrained controllability $Y_c$ analytically with \Cref{alg:ellipsoid-volume}, it can be inefficient for large-scale systems with high-dimensional state spaces. To address this challenge, we provide sufficient conditions under which
\begin{enumerate}
    \item adding constraints does not decrease controllability;
    \item constrained systems achieve the optimal unconstrained convex cost.
\end{enumerate}
More importantly, they can be verified efficiently. Our result is inspired by a recent work \cite{li_global_2023} on comparing localized/global model predictive control (MPC), with the critical difference that we require $x_T = 0$ in \eqref{eq:controllability_SLS}. While \cite{li_global_2023} requires that one must check the conditions hold for each $x_0$, we prove that if our conditions are satisfied, the performance is maintained for $x_0$ almost everywhere (a.e.).

First, we define sets $\cS(x_0)$ and $\T(x_0)$ as the set of all feasible states and control inputs under the dynamics~\eqref{eq:dynamics} and initial condition $x_0$, without and with constraints. Let 
    \begin{align*}
        Z_{AB0}\coloneqq \begin{bmatrix}
            I & & &\vline & 0 & & &\\
            -A & I & &\vline &-B & 0 & &\\
            & \ddots & &\vline & & \ddots & &\\
            & -A & I &\vline & & & -B & 0\\
            & & -A  &\vline & & & & -B
        \end{bmatrix}\,,
    \end{align*}
    such that $Z_{AB0}\in\R^{N_x T \times (N_x+N_u)T}$. Let $Z_p\coloneqq Z_{AB0}^\dagger \begin{bmatrix}
        I \\ 0
    \end{bmatrix}$, $Z_h\coloneqq I- Z_{AB0}^\dagger  Z_{AB0}$. Let $Z_h^{blk} = \text{blkdiag}(Z_h, \ldots, Z_h)$, repeated $N_x$ times. The indices $R$ are defined as the indices of $\vec \phi$ where the entry is constrained to be zero. Denote the matrix $F=Z_h^{blk}[R]$; due to the structure of $Z_h^{blk}$, $F$ has the form $F =\text{blkdiag}(\{Z_h[R_i]\}_{i\in[N_x]})$, where each $R_i$ is the subset of indices in $R$ that corresponds to the $i^{th}$ block of $Z_h^{blk}$. Define $x_0 \coloneqq [x_0[1]I\ x_0[2]I\ \dots\ x_0[N_x]I]$. We use the results of the following lemmas to prove the main results. Finally, let $H \coloneqq [H_1,\dots,H_{N_x}]$ where $H_i = I-F[R_i]^\dagger F[R_i]$.
\begin{lemma}\label{lem:Sx0}
    The dimension of $\cS(x_0)$ is given by $\dim \cS(x_0) = \rank{Z_h}\,.$
\end{lemma}
\begin{lemma}\label{lem:Tx0}
    The dimension of $\T(x_0)$ is given by $\dim \T(x_0) = \rank{Z_h x_0 (I-F^\dagger F)}\,.$
\end{lemma}
See the appendix for proofs, which follow closely from \cite{li_global_2023}.
\begin{theorem}[Optimality of Constrained Systems]\label{thm:optimality_constrained_systems} 
    For any constraint set $S_x$, $S_u$, if there exists a feasible solution $\phi\in \mathcal{S}$ satisfying \eqref{eq:SLS_dynamics}, $x_T=0$, and $\rank{Z_h} = \rank{Z_hH}$, then $\cS(x_0) = \T(x_0)$ for all $x_0$ almost everywhere.
\end{theorem}

\begin{proof}
    Note that $F$ is a block-diagonal matrix;
    the product $Z_h x_0 (I-F^\dagger F)$ is therefore
    \begin{align*}
        Z_h x_0 (I-F^\dagger F) = [
           Z_h  x_0[1](I - Z_h[R_1]^\dagger Z_h[R_1]) \  \dots \   Z_h x_0[N_x] ( I - Z_h[R_{N_x}]^\dagger Z_h[R_{N_x}])]\,.
    \end{align*}
    Let $a=\rank{Z_h H}$. The matrix $Z_h H$ can be row-reduced into a matrix $\hat H$, where the first $a$ rows of $\hat H$ are linearly independent and the last $N_x-a$ rows are linearly dependent on the first $a$ rows. Select $\lambda\in\R^a$ such that
    \begin{align*}
        \sum_{j=1}^a \lambda_j [\hat H_1[j],\dots,\hat H_{N_x}[j]] = 0\,.
    \end{align*}
    Then because the first $a$ rows of $\hat H$ are linearly independent, $\lambda_j=0$ for $j\in[1,\dots,a]$. For any $x_0$, define $\tilde x_0$ as $x_0$ re-weighted by the row-reduction used to define $\hat H$ from $H$. If $\tilde{x}_0$ satisfies $\tilde x_0[i]\ne0\ \forall\ i\in[N_x]$, then, for any $\tilde \lambda\in\R^a$ such that
    \begin{align*}
        \sum_{j=0}^a \tilde\lambda_j[\tilde x_0[1]\hat H_1[j],\dots,\tilde x_0[N_x] \hat H_{N_x}[j]] = 0\,,
    \end{align*}
    we have that $\tilde \lambda_j \tilde x_0[i]=0 $ for all $i\in[N_x]$ and $j\in[a]$. Because $\tilde x_0[i]\ne 0$ for all $i\in [N_x]$, so $\tilde \lambda_j=0$ for all $j\in[a]$. This results in the desired rank condition $\rank{Z_h x_0 H}=\rank{Z_h}$ for all $x_0$ such that $\tilde x_0[i]\ne 0$ for all $i\in[N_x]$. Note that we do not analyze the linearly dependent rows of $\hat H$ because they will remain linearly dependent when rescaled by $\tilde x_0$. By Lemmas~\ref{lem:Sx0} and \ref{lem:Tx0}, $\dim \cS(x_0)=\dim \T(x_0)$ for $x_0$ almost everywhere. Since $\T(x_0)\subseteq \cS(x_0)$, the same dimensionality implies that the sets are equal.
\end{proof}
Finally, we present a corollary of \Cref{thm:optimality_constrained_systems} that gives criteria for optimality conditions.
\begin{corollary}\label{cor:rank_condition}
    If there exists a feasible solution $\phi\in \mathcal{S}$ and $\rank{Z_h}=\rank{Z_h H}$, then for any convex cost function $f(x,u)$, the unconstrained cost is achieved under the constraints $\phi\in \mathcal{S}$ for $x_0$ almost everywhere. In particular, $V_c=Y_c^2$.
\end{corollary}
\begin{proof}
    By Theorem~\ref{thm:optimality_constrained_systems}, the theorem assumptions imply that $\cS(x_0)=\T(x_0)$ for $x_0$ almost everywhere. Therefore any optimization problem over the constrained set obtains the same value as the unconstrained problem for $x_0$ almost everywhere.
    
    Now select $f(x,u)=\norm{u}_2^2$. The above result implies that the optimization $h_c(x_0)$ in \eqref{eq:controllability_SLS} is over the same set with and without constraints, so the solution is the same for $x_0$ almost everywhere. Because the volume of an ellipse does not change over a measure-zero set, the volume over all $x_0$ is also the same. The result $V_c = Y_c^2$ then holds by Theorem~\ref{thm:reduction}.
\end{proof}

From this result, we conclude that for commonly-used cost functions such as quadratic cost, the addition of constraints does not necessarily hurt the performance of the controller.
This contrasts with some recent results on control \cite{shin2023near,lin2022decentralized} and reinforcement learning \cite{lin2021multi,zhang2023global}, where localized policies are less optimal than global policies. While this may seem surprising, it is worth noticing that we allow the controller to decide the (constrained) mapping $\vec{\phi}_u$ based on the initial state $x_0$, which is easier than committing to a policy before $x_0$ is sampled or picked by an adversary. Further, the SLS formulation allows us to consider a more general policy class $u(t)=K(t) x(t)$ than the static policy $u(t)=Kx(t)$ in \cite{shin2023near}. 

We show in the following numerical results how locality constraints satisfy this rank requirement and the controllability correspondingly does not change within this region. In contrast, actuation delay does not satisfy this condition and we observe a decrease in the controllability.
In addition to explaining the contrasting behavior of different types of constraints, computing the rank condition is faster and simpler to implement compared with computing the volume. We show a comparison in computation times in the numerical results.

\begin{remark}
    While the formulation of $h_c(x_0)$ might suggest that  one needs global knowledge of $x_0$ to solve $\phi_u$,
    one can show that when locality constraints are enforced, the value of each entry of $\phi_u$ depends only on the local entries of $x_0$. To see this, consider the following example: let $A$ be a tri-diagonal matrix (i.e., $A[i, j] = 0$ if $\abs{i-j}>1$), $B=I$, $T=1$ and the locality constraint to include only the immediate neighbors. Then $\supp \phi_u(t)$ is also tri-diagonal. Let $\phi_u$ denote $\phi_u(0)$ in this setting with only one time step. The minimization objective $\norm{\phi_u x_0}^2$ can be written as
    \begin{align*}
        \norm{\phi_u x_0}^2 &= 
            (\phi_u[0,0] x_0[0] + \phi_u[0,1] x_0[1])^2+\sum_{j=1}^{N_x-2}
            (\phi_u[j,j\tminus1] x_0[j\tminus1] + \phi_u[j,j] x_0[j] + \phi_u[j,j\text{+}1] x_0[j\text{+}1])^2 \\
            &\qquad + (\phi_u[N_x\tminus1 ,N_x\tminus2] x_0[N_x\tminus2] + \phi_u[N_x\tminus 1,N_x\tminus 1] x_0[N_x\tminus1])^2\,.
    \end{align*}
    Note that each $\phi_u[j]$ depends only on $x_0[j\tminus1]$ to $x_0[j\text{+}1]$. As this objective is \textit{row-separable} and the dynamical constraints are \textit{column separable}, $\phi_u$ can be computed locally using the alternating direction method of multipliers (ADMM), where at each iteration, a node only needs information from its direct neighbor. See \cite[Section 5.5]{anderson_system_2019} for details.
\end{remark}

\section{Numerical Results}\label{sec:numerics}

\subsection{Distributed Control of Linearized Swing Equations}
\begin{figure}
    \centering
    \includegraphics[scale=0.5]{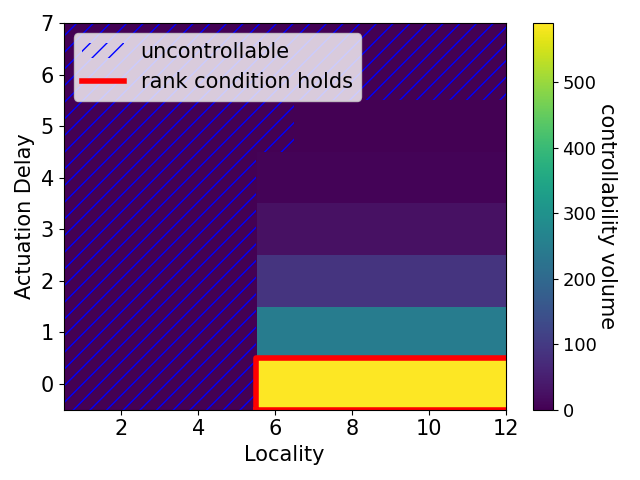}
    \caption{When actuation delay is zero, the addition of locality constraints on the system does not change the controllability volume for locality greater than five. The controllability decreases with increasing actuation delay.}
    \label{fig:power_system}
\end{figure}
To illustrate our results concretely, we consider the setting of the linearized swing equations, as used in \cite{li_global_2023}, which models power generation at multiple substations connected together in a power grid. This example allows us to illustrate how controllability changes based on actuation delay and locality constraints. The model is given by
{\small
\begin{align*}
    &x(t+1)[i]= \sum_{j:\G[i,j]=1} A_{ij} x(t)[j] + B_i u(t)[i] \\
    A_{ii} = &\begin{bmatrix}
        1 & \Delta t \\ -\frac{k_i}{m_i}\Delta t & 1-\frac{d_i}{m_i}\Delta t
    \end{bmatrix},\ \ \ 
    A_{ij} = \begin{bmatrix}
        0 & 0 \\ \frac{k_{ij}}{m_i} \Delta t & 0
    \end{bmatrix}\,,\ 
    B_i = \begin{bmatrix}
        0 \\ 1
    \end{bmatrix}.
\end{align*}}
The state $x(t)[i]=[\theta_i(t);\omega_i(t)]$ is the phase angle deviation and frequency deviation at node $i$. The parameters $k_{ij}$, $m_i^{-1}$, and $d_i$ are coupling, inertia, and damping parameters, with $k_{i}=\sum_{j:\G[i,j]=1}k_{ij}$. In the numerical implementation, the discretization step is $\Delta t= 0.1$, $T=12$, and there are nine two-state subsystems with one actuator per subsystem, resulting in $N_x=18$ and $N_u=9$. The subsystems are arranged in a $3 \times 3$ grid, with possible neighbors of a particular subsystem being the other subsytems to the left, right, top or bottom. The coefficents were sampled uniformly $k_{ij}\sim [0.5,1]$, $m_i^{-1}\sim[0,2]$, $d_i\sim[1,1.5]$.
In Figure~\ref{fig:power_system}, we observe that the settings where rank condition and feasibility hold coincides the maximum controllability; importantly, we observe no decrease in controllability with locality constraints, until the locality is too tight and the system is no longer controllable. The controllability decreases with increasing actuation delay, and as expect, the rank condition does not hold for any non-zero actuation delay. When the delay becomes large the system is no longer controllable.

\subsection{Computation of Rank Condition Versus Volume}\label{subsec:computation_speed}
The rank condition, which rigorously explains why communication speed and locality constraints may not result in decreased controllability, is easier to implement and faster to compute for large systems. We recorded the total time required to compute the controllability volume for all possible combinations of communication and locality constraints, for a dynamical system with the following parameters: the dynamics matrix $A$ is given by a tri-diagonal stochastic matrix, $B$ is filled with one and zero entries with actuation density $0.75$ relative to the state dimension, the time horizon is set to $T=10$, and $40$ trials were recorded for each combination of constraints. As seen in Figure~\ref{fig:time_comparison}, the time is comparable for $N_x$ ranging from $5-10$, but the volume computation becomes increasingly slower with larger $N_x$, with up to $6 \times$ speedup using the rank condition instead of computing the total volume. 
\begin{figure}
    \centering
    \includegraphics[scale=0.4]{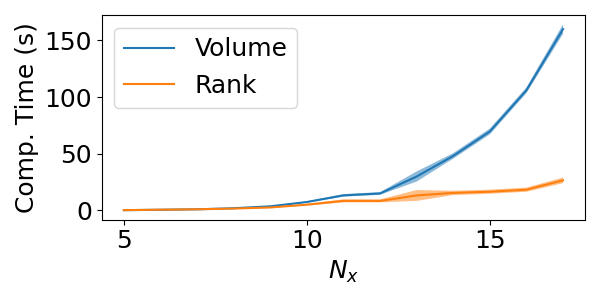}
    \caption{As the state dimension grows, the amount by which the rank condition computation time is faster than the controllability volume time increases. }
    \label{fig:time_comparison}
\end{figure}
 Users of our method can compute the controllability volume for the unconstrained system and then compute the rank condition and feasibility of the locality- and communication-constrained system to determine the controllability, which will either be equal to the unconstrained system or possible zero if the condition is not satisfied.

\section{Concluding Remarks}
This work proposes notions of constrained controllability $Y_c$ and constrained observability $Y_o$ using SLS, and shows that they reduce to the classic definitions for unconstrained systems. We provide closed-form solutions for computing $Y_c,Y_o$ and an efficient criteria to check whether a certain set of constraints can cause performance loss. In future work, we aim to apply this framework to solve real-world sensor and actuator problems, including air and water quality sensor placement for environmental monitoring. Additionally, we would like to investigate methods for computing the controllability in a distributed fashion for large-scale systems, which would greatly improve the range of settings in which this method can be used.

\appendix
    \section{Proof of Lemma~\ref{lem:Sx0}}
    To prove Lemma~\ref{lem:Sx0}, we first show that the set $\cS(x_0)$ can be written as 
    \begin{align*}
        \cS(x_0) = \{ 
        \textbf{s} :\textbf{s}=Z_p x_0 + Z_h x_0 \lambda\,,\ \lambda\in\R^{(N_u+N_x) N_x T} 
        \}\,,
    \end{align*}
    where $\textbf{s}=[\textbf{x};\textbf{u}]$. Using the dynamics and terminal constraint $x_T=0$, the closed-loop map $\phi$ can be written as
    \begin{align}\label{eq:cl_constraint}
        Z_{AB0} \phi = \begin{bmatrix}
            I \\ 0
        \end{bmatrix}\,.
    \end{align}
    This results in $\phi=Z_{AB0}^\dagger \begin{bmatrix}
            I \\ 0
        \end{bmatrix} + (I-Z_{AB0}^\dagger Z_{AB0})\Lambda$, where $\Lambda$ is a free variable. Using the definitions of $Z_p$ and $Z_h$ results in $\phi = Z_p + Z_h \Lambda$ . Then the state and input signals $\textbf{x}$ and $\textbf{u}$ result are given by $\textbf{s}=\phi x_0$, and we use the matrix $x_0$ to switch the multiplication of $\Lambda x_0$ to $x_0 \lambda$. The dimension of $\cS(x_0)$ is given by the degrees of freedom of $\textbf{s}$, which is $\rank{Z_h x_0} = \rank{Z_h}$ because $x_0$ is full row-rank by definition. 
\section{Proof of Lemma~\ref{lem:Tx0}}
    This proof proceeds similarly to that of Lemma~\ref{lem:Sx0}, except that now the constraint on $\phi$ is defined in terms of $\vec \phi$. From \eqref{eq:cl_constraint} and the additional requirements that $S_u\vec \phi_u =0$ and $S_x \vec \phi_x = 0$,
    \begin{align*}
        \begin{bmatrix}
            S_x \\ S_u
        \end{bmatrix} \left( \hat Z_p + \hat Z_h \vec \Lambda \right) = 0\ 
        \Rightarrow\  \hat Z_p[R] + \hat Z_h[R] \vec \Lambda\,,
    \end{align*}
    resulting in $\Lambda$ being parameterized as
    \begin{align*}
        \vec \Lambda = - \hat Z_h[R]^\dagger \hat Z_p[R] + (I-Z_h[R]^\dagger Z_h[R]) \mu\,,
    \end{align*}
    for a free variable $\mu$.
    Using the definition of $F$ and noting that $\hat Z_h[R]=Z_h^{blk}[R]$, the dimension of $\T(x_0)$ is given by $\rank{Z_h x_0 (I- F^\dagger F)}$.

\bibliographystyle{ieeetr}
\bibliography{references}

\end{document}